\documentclass{amsart}
\usepackage{graphicx} 
\usepackage{hyperref}
\usepackage{amsmath,amssymb,mathtools}

\theoremstyle{plain}
\newtheorem{theorem}{Theorem}[section]
\newtheorem{proposition}[theorem]{Proposition}

\theoremstyle{definition}

\theoremstyle{remark}
\newtheorem{remark}[theorem]{Remark}

\title[From Chaplygin sleigh to Torpedo]{From Chaplygin Sleigh to Torpedo: An~example of a~Cartan-Geometric View of Left-Invariant Nonholonomic Systems}

\author[Pavla Musilov\'a]{Pavla Musilov\'a}
\address[Pavla Musilov\'a]{Masaryk University, Faculty of Science, Department of Theoretical Physics and Astrophysics}\email{pavla@physics.muni.cz}

\author[Jan Slov\'ak]{Jan Slov\'ak}
\address[Jan Slov\'ak]{Masaryk University, Faculty of Science, Department of Mathematics and Statistics}\email{slovak@math.muni.cz}
\date{August 2026}
\thanks{The authors acknowledge the support of the Czech
Grant Agency project GA24-10887S, as well as the Horizon Europe MSCA projects CaLIGOLA and CaLiForNIA, ID 101086123 and ID 101119552.}


\begin{document}

\begin{abstract}
We discuss two models of a 3D torpedo in which the centre of mass $x$ moves only longitudinally, while the body may rotate about $x$. 
Although these are simple examples of left-invariant distributions on the Lie group $SE(3)$, they possess rich geometric structures. For the simplest model, we arrive at an abelian extension of a flat (3,5) path geometry, while adding tiny coupling, we obtain a non-trivial  hyperbolic type codimension two CR geometry whose symmetry algebra is just $\mathfrak{se}(3)$.
\end{abstract}

\maketitle

The so called Chaplygin sleigh is a typical simple example of left invariant non-holonomic mechanical systems, whose configuration spaces are Lie groups. There are powerful tools to handle such situations, both when looking for the mechanical behaviour or sub-Riemannian control theory optimisation, cf. \cite{Bloch, DGHM, AS}. 

The sleigh can be realized as a rod moving in the plane (a frozen lake) with a knife in its center of mass, allowing it to move only in the longitudinal direction. The configuration space clearly is $Q=\mathbb R^2\ltimes S^1=SE(2)$, i.e., the Lie group of Euclidean motions in the plane. By the constraint, there is only a 2-dimensional freedom of admissible velocities in the 3-dimensional tangent space to $Q$. Starting with a given momentum and position, the passive mechanical behavior follows the d'Alembert-Suslov principle. Mechanical kinetic energy is conserved, but the phase space volume compresses and the system passively drifts toward a stable straight-line trajectory, determined by the Euler–Poincaré–Suslov equations. 

This work deals with the straightforward 3-dimensional generalization, considering a solid rod, i.e., a cylinder $M$ in $\mathbb R^3$, with centre of mass $x$ allowed to move only in the longitudinal direction of its axis. Let us call such a mechanical system a \emph{torpedo}.

The configuration of such a system comes with a natural filtration and many low dimensional filtered manifolds come automatically equipped with a rigid natural Cartan geometry. This is also our case, and we shall focus on this phenomenon. In particular, we shall study the associated $(4,6)$ filtration on the configuration space, employing the full understanding of the relevant Cartan geometry picture, \cite{DeZanet, SchSl}. 

\section{Configuration Space and Kinematics}

According to the restrictions given, the velocity of the centre of the mass $x$ of our torpedo is collinear with its body axis and the cylinder may rotate about any axis through its centre $x$. Thus, the configuration space $Q$ of our torpedo $M$ is the Lie group
$$
Q = SE(3) = SO(3) \ltimes \mathbb{R}^3.
$$

The torpedo configuration is given by the choice of rotation $R\in SO(3)$ and position of $x\in\mathbb R^3$.  

Let $\omega \in \mathbb{R}^3\simeq \mathfrak{so}(3)$ denote the angular velocity and $v \in \mathbb{R}$ longitudinal velocity (speed) of the body frame. 

Let us write $A_1$, $A_2$, and $A_3$ for the left invariant vector fields generated by body rotation about the standard basis vectors $e_1$, $e_2$, and $e_3$, respectively, while $T_1$, $T_2$, and $T_3$ are the translations along these axes. These choices lead to the following structure constants:
\begin{equation}\label{SEbrackets}
[A_i,A_j] = \epsilon_{ijk}A_k, \quad [A_i,T_j] = \epsilon_{ijk}T_k,\quad [T_i,T_j] = 0
.\end{equation}

Our torpedo constraint says (viewing the longitudinal direction as the $e_3$ vector in the body fixed frame)
\begin{equation}
R^{-1}\dot x = v e_3    
\end{equation}
while any body angular velocity is admitted
\begin{equation}
R^{-1}\dot R = \omega
\end{equation}
with $\omega = \omega_1 e_1 + \omega_2 e_2 + \omega_3 e_3$. This yields
\begin{equation}\label{kinematic_eq}
    \dot R = R\omega, \quad \dot x = vRe_3
.\end{equation}
Rewriting this kinematics into the fixed body frame $(\omega_1,\omega_2,\omega_3,e_1,e_2,e_3)$, we arrive at the velocity $u\in\mathbb R^3$ constraint $u_1=u_2=0$, and $(\omega_1,\omega_2,\omega_3, u_3)$ remain free. 

Now, it is clear that this constraint does not depend on the configuration, and our mechanical system is left invariant. In the group unit $(E,0) \in SE(3)$, the distribution of admissible velocities is
\begin{equation}
    D= \operatorname{span}\{A_1,A_2,A_3,T_3\}
.\end{equation}
The two brackets $[A_1,T_3] = -T_2$ and $[A_2,T_3]=T_1$ provide the remaining two directions. Thus, $D+[D,D]=TSE(3)$ in all possible configurations.

We have arrived at the following

\begin{proposition}
The above defined 3D torpedo enjoys a non-holonomic left-invariant model on the configuration space $SE(3)$, with growth vector $(4,6)$, and driven by equations \eqref{kinematic_eq}.
\end{proposition}
\begin{proof}
As discussed above, the model is left-invariant. Thus, the filtration has the growth vector $(4,6)$, as we verified in the unit of $SE(3)$.
\end{proof}

\begin{remark} Our torpedo is a member of a family of geometries, starting with the Chaplygin sleigh as its planar case. Indeed, for a solid rod moving in $\mathbb R^n$ write $T_1,\dots, T_n$ for translations in directions of the standard basis vectors $(e_1,\dots, e_n)$, and $A_{ij}$ for the $\binom n2$ elementary rotations in the $e_i\wedge e_j$ planes generating $SO(n)$. Further, choose the left invariant distribution given at the unit $(E,0)$ as 
$$
D_n= \mathfrak{so}(n)\oplus \mathbb R T_n.
$$
Its rank is $\frac{n(n-1)}{2}+1$, while the dimension of $SE(n)$ is $\frac{n(n-1)}{2} + n$. 

Clearly, the brackets $[A_{in},T_n]=\pm T_i$, $1\le i< n$, provide the remaining $n-1$ dimensions and so the growth vector is at all points $(\binom{n}{2} +1, \binom{n+1}{2})$. 
Again, we obtain a left invariant one-step bracket generating distribution $D$ on the configuration space $SE(n)$. 
\end{remark}

Now, we come to the interesting special feature of the $(4,6)$ torpedo symbol. 
Set
\begin{equation}
    \mathfrak{g}_{-1} = D, \quad
    \mathfrak{g}_{-2} = \mathfrak{se}(3)/D = \operatorname{span}\{T_1, T_2\}
,\end{equation} 
so that the symbol of the filtered manifold $SE(3)$ is the left-invariant graded vector bundle generated by $\mathfrak{g}_{-2}\oplus\mathfrak{g}_{-1}$. The symbol Levi bracket is 
$$
\mathcal L:\Lambda^2\mathfrak g_{-1} \to \mathfrak g_{-2}
,$$
given by the brackets 
\begin{equation}\label{symbolbrackets}
\mathcal L(A_1,T_3)=-T_2,\quad \mathcal L(A_2,T_3)=T_1,\quad \mathcal L(A_i,A_j)=0,\quad \mathcal L(A_3,T_3)=0.    
\end{equation}

In particular, there is the characteristic line 
\begin{equation}\label{charline}
\operatorname{Char}(D)=\mathbb R A_3    
\end{equation} 
of the vector-valued 2-form $\mathcal L$. 
Clearly, for our torpedo, this is axial rolling around the longitudinal axis, which is not efficiently coupled in the system. 

In the general $SE(n)$ version, this characteristic subalgebra equals $\mathfrak{so}(n-1)\subset D_n$ and we do not see this phenomenon in the Chaplyging sleight planar case.

\begin{remark}
Actually, the more realistic mechanical picture would be to govern the rotations by an angular acceleration, i.e., an $\mathfrak{so}(3)$-valued torque (achieved, e.g., by some internal gyroscope inside our rod as usual in satellite steering). This means that we would add the rotation velocity to the configuration space parameters, $Q=SE(3)\times \mathfrak{so}(3)$, and consider the four independent steering control vectors
$$
X=T_3+\omega_1 A_1 +\omega_2 A_2 + \omega_3 A_3, \qquad U_i = \frac\partial{\partial\omega_i}.
$$
A direct check reveals that for $\tilde D=\operatorname{span}\{X,U_1,U_2,U_3\}$ we obtain
$$
\tilde D^2 = \tilde D + \operatorname{span}\{A_1,A_2,A_3\}
,$$
while $[A_1,T_3]= -T_2$ and $[A_2,T_3]=T_1$ yield the remaining two directions. Thus, the filtration has the growth vector $(4,7,9)$. 

Notice that with this fixed forward speed $X$ and three torgue controls $\omega_i$, the control vector field $X$ is a drift. Thus, our control system belongs to the so-called affine control systems, a more complicated object, cf. \cite{AS, Bloch}.
\end{remark}

\section{The Cartan geometry picture}
The Levi form on a distribution $D\subset TQ$ is the basic object for the Cartan geometry approach, and the relevant type of geometry is revealed from the orbit of the symbol under the natural automorphism group of the filtrations. 

In our 3D torpedo case, we must hit one of the orbit types classified in De Zanet's work, \cite{DeZanet}. This paper deals with one-step generated filtered manifolds defined by a distribution $D$ of dimension 4. 

Notice that the wedge product defines a conformal quadratic form $b$ of signature $(3,3)$ on $\Lambda^2 D$ and De Zanet showed that the orbits are fully determined by the restriction of the quadratic form to the kernel of the Levi bracket $\mathcal L$, which is of dimension 4, namely by its rank and signature. Equivalently, we may use the 2-plane $W = \operatorname{Im}\mathcal L^* \subset \Lambda^2 D^* \simeq \Lambda^2 D$, i.e., $W$ is the annihilator of the kernel, and we may consider the restriction of the quadratic form to $W$. The De Zanet's classification of the orbits is then expressed by the signatures $(2,0),~(1,1),~(1,0),~(0,0)$ of $b_{|W}$.  

In our case, $D=\operatorname{span}\{A_1,A_2,A_3,T_3\}$, $TQ/D\simeq\operatorname{span}\{T_1,T_2\}$, and the symbol is given by the brackets \eqref{symbolbrackets}. Thus, writing $e^1,e^2,e^3,e^4$ for the dual basis to $A_1,A_2,A_3,T_3$,
$$
W = \operatorname{span}\{e^1\wedge e^4, e^2\wedge e^4\}
,$$
and clearly $b_{|W}=0$, i.e., our quadratic form vanishes on $W$. The symbol of our torpedo belongs to the orbit of totally isotropic two-planes in the space $\Lambda^2\mathbb R^4\simeq \mathbb R^{3,3}$.   
Consequently, our torpedo belongs to the most degenerate of possible $(4,6)$ geometries, while the two $(4,6)$ geometries studied in detail in \cite{SchSl} belong to the two open orbits corresponding to the two regular quadric types. Actually, it is also clear that our torpedo cannot be one of those codimension 2 CR-geometries of \cite{SchSl}, since this would require the existence of compatible almost complex structures on $D$. Since the characteristic line of $\mathcal L$ would have to be invariant under any Levi compatible complex structure on $D$, this is impossible for a one-dimensional subspace.

Let us move towards the Cartan picture of our torpedo, now. Notice that the characteristic line $\operatorname{Char}D\subset D$ brings an infinite-dimensional space of infinitesimal symmetries. Indeed, for any smooth function $f$ and $X\in D$, $[fA_3,X]=f[A_3,X] - X(f)A_3\in D$. Thus, let us quotient out the non-coupled axial rotation $A_3$, i.e., consider the quotient by the characteristic foliation
$$
p:Q=SE(3)\to \bar Q=SE(3)/SO(2)
.$$
We may identify $SE(3)/SO(2) \simeq \mathbb R^3\times S^2$, with $p(R,x) = (x, n)$, where $n=Re_3$. This corresponds to claiming that we do not see the rotations of the torpedo about its axis.

The original distribution $D$ descends to rank 3 distribution
$$
\bar D = D/\operatorname{Char}D \subset T(\mathbb R^3\times S^2), \bar D_{(x,n)}= \bigl\{(\dot x,\dot n);\ \dot x = u n,\ \dot n \in T_n S^2\bigr\}
$$
and this distribution on $\bar Q$ naturally decomposes as 
$$
\bar D = E\oplus V,\qquad E = \operatorname{span}\{(n,0)\},\quad V=\{0\} \oplus T_nS^2
.$$
Now, the Levi bracket provides a linear isomorphism $E\otimes V\to T\bar Q/\bar D$, while $[V,V]=0$, $[E,E]=0$. We are arriving at 
\begin{theorem}
The characteristic reduction torpedo model $\bar Q$ is the oriented correspondence space $\bar Q=\mathbb R^3\times S^2$ equipped with a path geometry structure. After further $\mathbb Z_2$ reduction reversing the longitudinal axis direction, it becomes $\mathbb P(T\mathbb R^3)$ with the standard flat path geometry of type $(3,5)$.  
\end{theorem}

\begin{proof}
The above described natural splitting $\bar D=E\oplus V$ defines the path geometry on $\bar Q$, see \cite[Section 4.4.3]{parabook}. So it remains to prove that this geometry is double cover of the flat path geometry on $\mathbb P(T\mathbb R^3)$.

Consider (oriented) directions $n$ with non-zero third coordinate, normalize them as $n\sim (p,q,1)$, and consider the local coordinates 
$(x^1,x^2,x^3,p,q)$ on $\mathbb R^3\times S^2$. Further consider the forms $\theta^1 = dx^1 - p dx^3$, $\theta^2 = dx^2 - qdx^3$. Then $\bar D = \operatorname{ker}\{\theta^1, \theta^2\}$ and $\bar D$ is spanned by 
$$
V_1 = \partial_p,\quad V_2 = \partial_q,\quad X=\partial_{x^3} + p\partial_{x^1} + q \partial_{x^2}
.$$
Thus, $V = \operatorname{span}\{V_1,V_2\}$, $E = \mathbb R X$, $[V_1,X] = \partial_{x^1}$, $[V_2,X]=\partial_{x^2}$.
Since the integral curves of $X$ project to  (oriented) straight lines in $\mathbb R^3$, 
we arrive at the flat path geometry of projective lines in $\mathbb R^3$ as proclaimed.
\end{proof}

\begin{remark}
In view of our reduction, the original torpedo model is the $SO(2)$ principal bundle $p:SE(3)\to \bar Q = \mathbb R^3\times S^2$. 
In order to view it as a finite type Cartan geometry we need to restrict the infinite dimensional symmetries related to the uncoupled rotations $A_3$. A natural way to do this is to look for a natural connection on this bundle as follows: 

As before, we understand $\omega=(\omega_1,\omega_2,
\omega_3)$ as the $SO(3)$ part of the Maurer Cartan form on $SE(3)$, i.e., $\omega=R^{-1}dR$. 
Now, we may view the form $\nu=\omega_3$ as a reasonable choice of a principal $SO(2)$ connection associated to the reductive splitting $\mathfrak{so}(3)=\operatorname{span}\{A_3\}\oplus \operatorname{span}\{A_1,A_2\}$, and its curvature is $d\nu=-\omega_1\wedge\omega_2$. Further, $D\cap \operatorname{ker}\nu = \operatorname{span}\{A_1,A_2,T_3\}$ is clearly the horizontal lift of the path geometry distribution $\bar D$.

The flat path geometry is given by the partial flag variety $SL(4,\mathbb R)/P_{12}$ with the parabolic subgroup $P_{12}$ corresponding to indicences of lines in planes in $\mathbb R^3$. The $\mathbb Z_2$ quotient yields the $SO(2)\times \mathbb Z_2$ principal fiber bundle $\pi:Q\to \mathbb P(T\mathbb R^3)$. 
Packing all this reduced geometry data into one Cartan connection, we arrive at the following theorem.
\end{remark}

\begin{theorem}
The 3D torpedo geometry including the above choice of  the $SO(2)$ connection $\nu$ is an abelian extension of the locally flat path parabolic geometry of type $(SL(4,\mathbb R),P_{12})$. Its Cartan connection on the bundle $\pi^*SL(4,\mathbb R)\times SO(2)\to Q$ is $\alpha = (\pi^*\beta, \nu)$, where $\beta$ is the flat path geometry connection on $\mathbb P(T\mathbb R^3)$. The curvature $\Omega$ of $\alpha$ is $\Omega=(0, -\omega_1\wedge\omega_2)$ and it is constant. This curvature describes the roll holonomy around the torpedo body axis.

The symmetry algebra of this geometric model is the sum of the algebra of conformal Euclidean motions $\mathfrak{cso}(3) \rtimes \mathbb{R}^3$ and $\mathfrak{so}(2)$. Thus 8-dimensional. 
\end{theorem}

\begin{proof}
The above construction and consideration can be summarized exactly as done in the theorem. So we have got the Cartan connection $\alpha$. Now, we have to compute the symmetry algebra. 

Let us first look at simple component, the $SO(2)$ fibers. Here, a change of gauge by right multiplication by $\phi(x,n)$ changes the connection as $\theta \mapsto \theta + d\phi$ and so any symmetry of the original torpedo survives only if $d\phi$ vanishes. Thus the infinite-dimensional gauge freedom reduces to one copy of $\mathfrak{so}(2)$. 

Next, the curvature $\omega_1\wedge\omega_2$ of our Cartan connection $\alpha$ descends to a 2-form $\Omega$ on the basis $\bar Q=\mathbb R^3\times S^2$, and its restriction to the vertical planes $V\subset T(\bar Q)$ define the standard spherical area form. Moreover, $E$ is the kernel of $\Omega_{|\bar D}$. Thus, our choice of the distribution $D\subset TQ$ and $\nu=\omega_3$ determines on $\bar Q$ the flat path geometry and the spherical area form on the direction fibers. Both have to be preserved by any symmetry and, consequently, the symmetry algebra is strongly reduced. 

Locally, the automorphisms of the flat path geometry arise from projective transformations. Writing $A_x$ for the derivative of such a transformation, its action on an oriented direction $n\in V$ is $n\mapsto (1/|A_x n|)A_x n$. Preservation of the area form requests its Jacobian $J(n) = \operatorname{det} A_x/|A_x n|^3$ equal to one. Thus, $|A_x n|$ must be independent of $n$ and so $A_x = \lambda(x)B(x)$, where $B(x)\in SO(3)$. So the transformation must be simultaneously conformal and projective, thus it will be a conformal affine transformation.

Lifting everything up, back to $Q$, we obtain the following symmetries: \\ spatial translations 
$$
(x,R)\mapsto (x+a,R),\quad a\in\mathbb R^3,
$$
spatial dilation
$$
(x,R)\mapsto (\lambda x,R),\quad \lambda\in \mathbb R, \ \lambda>0 
,$$
spatial rotation 
$$
(x,R) = (Ax,AR),\quad A\in SO(3)
,$$
and the constant body roll
$$
(x,R)\mapsto(x,R\cdot\phi), \quad \phi\in SO(2)
.$$
Notice there are two different rotations. The left multiplication by elements in $SO(3)\subset SE(3)$, and the right multiplication by an element in $SO(2)$ which is an internal body-roll symmetry.
\end{proof}

\begin{remark}
Actually, much of the information from the latter theorem can be seen directly from the Maurer-Cartan equations on $SE(3)$. Let us write the center of the mass body translational forms $\eta_i$ in coordinates:
$$
R^{-1}dx = \eta_1 e_1 + \eta_2 e_2 + \eta_3 e_3
$$
so that $D= \operatorname{ker}\eta_1\cap \operatorname{\eta_2}$ and the structure equations include
\begin{align}
d\eta_1 &= -\omega_2\wedge \eta_3 + \omega_3\wedge \eta_2
\\
d\eta_2 &= \omega_1\wedge\eta_3 - \omega_3\wedge\eta_1
\\
d\omega_3 &= -\omega_1\wedge\omega_2
\end{align}
The first two equations, modulo $\eta_1$ and $\eta_2$, describe the totally isotropic Levi bracket pencil, while the third one is invisible in the symbol (since $[A_1,A_2]=A_3\in D$).

With a little effort we may check that the chosen principal $SO(2)$ connection leaves the biggest symmetry for the model, once we request all the $SE(3)$ motions to remain among the symmetries. Of course, locally we might choose a flat $SO(2)$-connection there, which would keep the entire 15-dimensional symmetry algebra of the flat path geometry there. But there is no globally defined flat connection available. All other connections built of the forms $\omega_i$ restrict the symmetries at least as much as our choice $\nu=\omega_3$.  
\end{remark}

\section{A hyperbolic model torpedo}

Clearly, the reason for the special behaviour of the standard 3D torpedo model was caused by the uncoupled rotation about the body axis. Let us imagine we would add two orthogonal fin constraints placed at the torpedo body at the distance $\ell$ from its center, let us say at the head and the tail of it. 

If the center reference point is $x$ and the body axis is $e_3$ as before, then the head and tail of the torpedo are $x_+ = x + \ell e_3$, $x_-=x-\ell e_3$. Taking into account the rotation velocity $\omega$, the velocity of the head and tail points would be 
$$
v+\omega\times (\pm \ell e_3),
$$
where $v$ is the velocity of the reference point $x$. Now, let us impose the restriction, that the tail of the torpedo is not allowed to move in the $e_1$ direction (by mounting the appropriate fin there), while the head is not allowed to move in the $e_2$ direction. This leads to the following conditions $e_1\cdot (v + \omega\times (- \ell e_3))=0$, $e_2\cdot (v + \omega\times (\ell e_3))=0$, implying
\begin{equation}
v_1 = \ell \omega_2, \qquad v_2 = \ell \omega_1
.\end{equation}
The third velocity component $v_3=u$ remains free. If $\ell=0$, we obtain the previous model, but with $\ell > 0$ and the standard controls $(u,\omega_1, \omega_2, \omega_3)$ we arrive at the more fancy generators of our control distribution $D_\ell$:
\begin{equation}
D_\ell = \operatorname{span}\{X_1=A_1+\ell T_2, X_2 = A_2 + \ell T_1, X_3 = A_3, X_4 = T_3\}
.\end{equation}
Again, we formulated our constraints as fixed to the body frame, and so the control distribution is left invariant on $SE(3)$. The brackets \eqref{SEbrackets} lead, modulo $D_\ell$, to the symbol brackets
\begin{equation}
 [X_1,X_3] = 2\ell\bar T_1,\quad [X_2,X_3]=-2\ell \bar T_2,\quad [X_1,X_4]=-\bar T_2,\quad [X_2,X_4]=\bar T_1   
,\end{equation}
while $[X_1,X_2]\in D_\ell$ and $[X_3,X_4]=0$.

Now, we can repeat the discussion of the type of the symbol in the De Zanet's classification. In the dual basis $e^1,e^2,e^3,e^4$ of $D_\ell^*$, the image $W$ of $\mathcal L^*$ is now spanned by the forms 
$$
\mu_1 = 2\ell e^1\wedge e^3 + e^2\wedge e^4,\qquad \mu_2 = -e^1\wedge e^4 - 2\ell e^2\wedge e^3.
$$ 
The Pfaffian of their pencil $a\mu_1 +b\mu_2$ is $2\ell(b^2-a^2)$ (up to sign conventions) and thus we get two distinct roots $a=b$ and $a=-b$, always one positive and one negative. Thus, we have proved: 

\begin{theorem}
The modified 3D torpedo model with the two orthogonal fin constraints at its head and tail provides a hyperbolic $(4,6)$ type geometry.
\end{theorem}

\begin{remark}
A natural question is, how to deform our control distribution $D$ of the isotropic model above to get the other open orbit type in the De Zanet's classification. The answer is suprisingly simple: no way!

Indeed, let us consider a general coupling of the rotations $\omega$ and center body velocities $v$. We keep $v_3=u$ to be free, while we allow 
$$
v_1 = \alpha\omega_1+\beta\omega_2, \qquad v_2 = \gamma\omega_1 +\delta\omega_2
.$$
Then a direct computation with the two forms $\mu_1$ and $\mu_2$ that span the image of $\mathcal L^*$ as above, yields the Pfaffian of their pencil in the form
$$
(\beta+\gamma)(a^2-b^2) + 2(\alpha-\delta)ab
.$$
Its discriminant is $4((\alpha-\delta)^2 + (\beta+\gamma)^2)$ and thus, we land in an open orbit if and only if $(\alpha-\delta,\beta+\gamma)\ne (0,0)$ and the discriminant is always non-negative. This shows that we always land either in the hyperbolic orbit or in one of the non-open ones. 

If we understand the coupling of the translational and rotational velocities as a linear map
$$
C = \begin{pmatrix} \alpha & \beta \\ \gamma & \delta\end{pmatrix}
,$$
then the degenerate orbit condition is that 
$$
C = \alpha\begin{pmatrix} 1 & 0 \\ 0 & 1\end{pmatrix} + \beta\begin{pmatrix} 0 & -1 \\ 1 & 0\end{pmatrix}
.$$
Consequently, we get a hyperbolic $(4,6)$ type geometry if and only if the coupling matrix contains a nonzero symmetric trace-free part. Our simple example above comes with 
$$
C = \ell\begin{pmatrix} 0 & 1 \\ 1 & 0\end{pmatrix}
,$$
symmetric and trace-free.
\end{remark}

\section{The codimension two CR picture}

The latter theorem invites us to fully understand our coupled torpedo model in the realm of special codimension two CR structures of \cite{SchSl}. But we have to beware that the hyperbolic $(4,6)$ geometry itself is an infinite type geometry, while we need to reduce the structure group by the choice of an almost complex structure $J$ on $D_\ell$ to get a genuine CR structure of the given type there.

Fortunately, we are going to see that there is a natural choice for such a structure, once we consider a sub-Riemannian metric on $D_\ell$. 

In order to simplify our formulae, let us declare
$$
E_1=A_1+\ell T_2,\ E_2 = A_2 + \ell T_1,\ E_3 = A_3,\ E_4 = 2\ell T_3
$$
to be the orthonormal frame of $D$. Then, choosing $T_1$ and $T_2$ as the basis of the complement, the Levi form is determined by the two 2-forms
$$
\mu_1 = e^1\wedge e^3 + e^2\wedge e^4,\qquad \mu_2 = -e^1\wedge e^4 - e^2\wedge e^3
$$
up to the common factor $2\ell$.

As well known from \cite{SchSl}, the hyperbolic codimension two CR structures come with an almost product structure on the entire $TM$, compatible with the Levi form and a compatible almost complex structure on the distribution $D$, see also \cite[Section 4.3.8]{parabook} 

Let us start with the almost product structure. In the pencil of the two forms $\mu_1$ and $\mu_2$, there are two decomposable ones:
\begin{align*}
\mu_1+\mu_2 &= (e^1-e^2)\wedge(e^3-e^4)
\\
\mu_1-\mu_2 &= (e^1+e^2)\wedge(e^3+e^4) 
.\end{align*}
Their kernels are 
\begin{align*}
D_+&=\operatorname{span}\{E_1+E_2,E_3+E_4\}
\\
D_-&=\operatorname{span}\{E_1-E_2,E_3-E_4\}
.\end{align*}
This is the natural splitting, up to swapping $D_\pm$. The corresponding almost product structure $P$ can be defined as $P_{|D_+}=\operatorname{id}_{D_+}$, $P_{|D_-}=-\operatorname{id}_{D_-}$. Thus, $P^2=\operatorname{id}_D$ and $\operatorname{Tr}P=0$. Clearly, $P(E_1)=E_2$, $P(E_2)=E_1$, $P(E_3)=E_4$, $P(E_4)=E_3$.
Thus, our affinor $P$ exchanges the two transverse rotation controls and simultaneously exchanges axial roll with normalized longitudinal propulsion.

The splitting into left and right components is exactly the hyperbolic structure considered in \cite{SchSl}. In fact the two components may be exchanged by a discrete symmetry, so in general the splitting is labelled only locally, but for our left-invariant model on $SE(3)$, the body frame labels it globally.

As also known from \cite{SchSl}, this splitting must extend to the quotient $U=TSE(3)/D$ as well. Indeed, taking 
$$ 
U_+=\operatorname{span}\{\bar T_1-\bar T_2\}, \qquad U_-=\operatorname{span}\{\bar T_1+\bar T_2\}\},
$$
a direct computation reveals
$$
[D_+,D_-]\bmod D = 0 
$$
while
$$ [D_+,D_+]\bmod D \subset U_+, \qquad [D_-,D_-]\bmod D \subset U_- 
.$$
Thus, the graded tangent algebra splits as
$$ \operatorname{gr}(TM) = (D_+\oplus U_+) \oplus (D_-\oplus U_-). 
$$
Each summand is a three-dimensional Heisenberg symbol and this is the algebraic meaning of De Zanet’s statement that the hyperbolic \((4,6)\) model Lie algebra is the product of two real three-dimensional Heisenberg algebras.

While the almost product structure on the symbol algebra is intrinsic, the choice of the almost complex structure $J$ on $D$ is not. We need to fix a compatible complex structure $J$ on $D$, i.e., making $D_\pm$ into complex lines and then, following the general theory in \cite{SchSl}, there will be the canonical normal Cartan connection there, determining also the natural split of the entire $TM$ into $T_\pm M$, i.e. the natural extension of the almost product structure $P$ to the entire $TM$.

The Levi bracket restricted to each hyperbolic factor gives a nondegenerate line-valued two-form

$$ \mathcal L_\pm: \Lambda^2D^\pm\to U^\pm. 
$$
Choosing the orientations of $U^\pm$, we turn $\mathcal L_\pm$ into area forms $ \Omega_\pm\in\Lambda^2(D^\pm)^*$ on the 2-dimensional spaces, up to a choice of scale.

Next, choose positive metrics $g_\pm$ on \(D^\pm\), as we did by our choice of the orthonormal basis $E_i$. 
Now we normalize $\Omega_\pm$ to the $g_\pm$-area forms, and we define \(J_\pm\) by
$$ g_\pm(J_\pm X,Y)=\Omega_\pm(X,Y). $$
Then clearly
$$ J_\pm^2=-1, \qquad J=J_+\oplus J_-.$$

Declaring our frame $E_i$ to be orthonormal means that
the two transverse rotational controls have equal cost; axial roll has the same normalized cost; and thrust is measured in the length scale $c=2\ell$.

We orient our two orthonormal subspaces $D_\pm$ using basis
$$ X_\pm=E_1\pm E_2, \qquad Y_\pm=E_3\pm E_4. $$
and define
$$ J X_\pm=Y_\pm, \qquad J Y_\pm=-X_\pm. $$

Equivalently, directly in our \(E_i\)-frame,
$$\begin{aligned} JE_1&=E_3,&\qquad JE_3&=-E_1,\\ JE_2&=E_4,& JE_4&=-E_2. \end{aligned} $$
Obviously, $J^2=-\operatorname{id}_D$ and, moreover, $JP_H=P_HJ$. Thus, 
\(J\) preserves both hyperbolic factors, i.e. comutes with the almost product structure $P$, and each $D_\pm$ becomes a complex line.

The essential CR compatibility condition is
$$\mathcal L(JX,JY)=\mathcal L(X,Y),$$
and this is checked easily for our Levi form.

Summarizing, we see what is intrinsic and what is the additional mechanical input:
the Levi symbol bracket itself induces $D_\pm$, $U_\pm$, $P$; the control metric adds $J_\pm$; body orientation and labelling of our fin coupling provides our signs and ordering.

This completes the initial data for the hyperbolic codimension two CR structure on our torpedo model and, finally, we want to see how far is this geometry from the flat model. From \cite{SchSl} we know that the harmonic curvature decomposes into 6 torsions and 2 curvatures, all of them one-dimensional and coming in couples. Let us discuss them
one by one. 

First of all, the normalization enforces the choice of the right complement of $D$ in $TM$. Once we get this, we may compute all the torsion components easily. 

We shall use our generators$X_\pm=A_1\pm A_2 +\ell(T_2\pm T_1)$, $Y_\pm=A_3\pm 2\ell T_3$ of $D_\pm$.
Following \cite{SchSl}, the canonical three-planes $T_\pm$ must satisfy
$[D_\pm,D_\pm]\subset T_\pm$.
In our case 
$$
[X_\pm,Y_\pm]=\pm X_\mp + 4\ell(T_1 \mp T_2)
$$
and thus correct generators for the complements of $D_\pm$ in $T_\pm$ are 
$$ Z_+=T_1-T_2+\frac{1}{4\ell}X_-,\qquad  Z_-=T_1+T_2-\frac{1}{4\ell}X_+
.$$

As described in \cite{SchSl}, the three pairs of the torsion are $N_\pm$ 
expressing the Nijenhuis tensor of $J_\pm$; the so called $S$-torsions, given by the formula 
$$S_\pm (X,Y) = \pi_\pm([X,Y] + [JX,JY] - J[JX,Y]  +J[X,JY])$$ with $X\in D_+$ 
and $Y\in D_-$; and the almost-product torsion $\tau_\pm$ expressing the integrability of the almost-product structure on $TM$. 

A straightforward check reveals that the Nijenhuis tensors as well as the $S$-torsions of our torpedo model vanish. The almost product obstruction torsion $\tau$ is non-trivial (we have to compute in the symbol algebra):
$$
\tau_+(\bar Z_+,X_+)=0,\quad \tau_+(\bar Z_+,Y_+)=-\bar Z_-,\quad 
\tau_-(\bar Z_-,X_-)=0,\quad \tau_-(\bar Z_-,Y_-)=\bar Z_+.
$$

In summary, our torpedo model is an integrable CR structure whose $S$ torsion vanishes too, but it is not a product of two hypersurface type $CR$ structures of dimension 3. Indeed, recall that following \cite{SchSl}, it would be such a product if and only if all its 6 harmonic torsion components vanish.

Next, we want to compute the harmonic curvature. The two components are represented by chains
$$
W_\pm:\mathfrak g_{-2}^\pm\times\mathfrak g_{-1}^\pm\to \mathfrak{g}_1^\pm
.$$
If the model was torsion-free, then $W_{\pm}$ would be exactly the two Chern-Moser curvatures of the 3-dimensional components.

In our case, since the model is left invariant, the curvatures must be constant, thus given by their scalar value in one point.

Moreover, since our simple fin based coupling enjoys the discrete symmetry exchanging the two null Levi factors and thus the two curvatures will be the same scalars, up to sign or complex conjugation, depending on our choices of the orientations etc.

Because our configuration space $M$ is the Lie group \(SE(3)\), the normal Cartan connection can be encoded by a linear map
$$ \alpha: \mathfrak{se}(3) \to \mathfrak g = \mathfrak{su}(2,1)_+\oplus\mathfrak{su}(2,1)_-$$
whose projection to \(\mathfrak g/\mathfrak p\) is our adapted crossed-fin frame.

Its curvature is
$$ K_\alpha(X,Y) = [\alpha(X),\alpha(Y)] - \alpha([X,Y])$$
and the normality condition reads
$$ \partial^*K_\alpha=0.$$
We have to solve this equation 
successively in homogeneities \(1,2,3,4\). The calculation must produce:
the already computed pair $(\tau_+,\tau_-)$ in degree 1;
the non-harmonic higher components determined by the normalization;
the remaining harmonic components $W_\pm$ in degree 4.

Thus, there are no differential equations, we are left with finite-dimensional linear algebra, followed by a small quadratic computation. This quite tedious computation leads to the final conclusion
$$W_\pm=\mp\frac{15}2.$$

We have arrived at the final result
\begin{theorem}
The cross-fin based 3D torpedo model enjoys a natural homogeneous left-invariant hyperbolic integrable codimension two CR geometry which is not a product of two hypersurface-type CR geometries and has nontrivial constant torsion and curvature.
The symmetry algebra of this model is $\mathfrak{se}(3)$.
\end{theorem}

\begin{proof}
We have derived all the claims of the theorem, except the last one. 
Of course, the symmetry algebra of the left-invariant model must include the entire $\mathfrak{se}(3)$, and it clearly cannot be the full 16-dimensional algebra of the flat model.

For our cross-fin geometry, \(SE(3)\) already acts transitively by symmetries. Therefore, any symmetry beyond these six dimensions would have to contribute a nontrivial isotropy symmetry in the graded vertical subalgebra at a chosen point \(x\).

After choosing a point \(u\) in the Cartan bundle over \(x\), the isotropy is filtered, and its associated graded algebra embeds into
$$ \mathfrak g_{\ge 0} = \mathfrak g_0\oplus\mathfrak g_1\oplus\mathfrak g_2. $$

The standard Tanaka procedure then starts with the degree-zero part
$$ \mathfrak a_0 = \operatorname{ann}_{\mathfrak g_0}(K_H(x))$$
where \(K_H\) is the harmonic curvature. Since every infinitesimal symmetry preserves curvature,
$$ A\cdot K_H(x)=0 $$
for its degree-zero isotropy component \(A\).

In our case, a straightforward computation shows that $\operatorname{ann}_{\mathfrak g_0}(\tau_\pm)$ already vanishes. 
Thus, the Tanaka-prolongation algorithm cannot produce any nontrivial additional symmetry.

Notice that the two non-trivial torsion components themselves killed all additional symmetries, so we did not have to look at the curvatures at all.
\end{proof}


\end{document}